\documentclass[a4paper,11pt]{article}

\usepackage{amsthm}
\usepackage{amsfonts}
\usepackage{amsmath}
\usepackage{amssymb}
\usepackage{graphicx}
\usepackage{epstopdf}
\usepackage{cite}
\usepackage{graphicx}
\usepackage{caption}
\usepackage{subcaption}
\usepackage{booktabs}
\usepackage{multirow}
\usepackage{url}
\usepackage{authblk}

\newtheorem{thm}{Theorem}

\newtheorem{prop}[thm]{Proposition}
\begin{document}

\title{Capacity Expansion Planning of Wind Power Generation in A Market Environment with Topology Control}


%\titlerunning{Short form of title} % if too long for running head

%\author{Yifan Wang , Bo Zeng  \thanks {Corresponding author} , Shixin Liu}

%\title{More than one Author with different Affiliations}
\author[1]{Yifan Wang \thanks{santayifan@hotmail.com}}
\author[2]{Bo Zeng \thanks{bzeng@pitt.edu}}
\author[1]{Shixin Liu \thanks{sxliu@mail.neu.edu.cn}}

\affil[1]{College of Information Science and Engineering, Northeastern University, Shenyang, Liaoning, China}
\affil[2]{Department of Industrial Engineering, University of Pittsburgh, Pittsburgh, PA, USA}

\date{December, 2016}
\maketitle

\textbf{Abstract}: Wind power integration is an essential problem for modern power industry. In this paper, we develop a novel bilevel mixed integer  optimization model to investigate wind power generation planning problem in an electricity market environment with topology control operations. Different from existing formulations, the lower level market clearing problem introduces binary variables to model line switching decisions, which have been proven very effective to improve transmission capability under different load profiles. To solve this challenging bilevel MIP, a recent decomposition method is customized and a couple of enhancement techniques based on grid structure are designed. Through computing instances from typical IEEE test beds,  our solution methods demonstrate a strong solution capacity. Also, we observe that applying topology control on a small number of lines could be very helpful to reduce wind power curtailment and improve the wind penetration level.


\section{Introduction}
\label{intro}

\subsection{Background and Motivation}
Wind power generation has become a primary renewable energy source since 2000. It actually is treated as a most effective approach by energy industry to reduce greenhouse gas emission, to mitigate climate change, and to achieve sustainability. Now, almost all major countries have established long-term goals to promote the development of wind power generation facilities and the integration of wind power into power systems. For example, the U.S. Department of Energy (DOE) has set a target that 20\% electricity supply in 2030 should be from wind generation \cite{lindenberg200820}.  To address its serious air pollution issue  and also to support the increasing demand on electricity, China has considered wind power is the most critical component in its energy development plan. A recent report shows that its goal is to generate 25\% total electricity from wind power by 2030 \cite{davidson2016modelling}.

To achieve those goals, the development of wind power generation has been grown up drastically. As shown in Figure \ref{fig:fig1}, the cumulative installed wind power generation capacity of the world grows from 17,400 MW in 2000 to 432,833 MW in 2015. North American and European Union countries were the pioneers on wind energy development and utilization before they were replaced by China in 2009. Up to 2015, China now has more than 145 GW of wind power capacity installed, which is more than the total of the European Union \cite{GWEC} and is twice as much as of that in the United States. Actually, the fast development of wind power in China has not been slowed down yet. According to the information released by National Energy Administration of China  \cite{NEA}, the newly installed wind power generation capacity in 2015 is about 32.97 GW, which is more  than 20\% of its total capacity in 2015.

Although the commercial-scale wind power generation facilities have been widely deployed, it still remains as a challenging issue to effectively integrate wind power into power grid. Basically, there are two reasons. One is that the nature of wind power, i.e., its variability, intermittency, and less controllable characteristics, make it less compatible and dispatchable in a grid, where the whole system should be highly reliable. We note that such difficulty can be mitigated by
utilizing better hardware, e.g., more flexible generation assets, and
software, e.g., sophisticated scheduling and management methods, to improve the controllability of wind power generation. Another one is that the long physical distances between wind resources and load centers require for strong power delivery systems, which are often insufficient or congested to absorb wind power. As shown in  Figure \ref{fig:wr_pd}, the most resourceful areas in China are its northern and western parts while most people are living in the southern and eastern parts. Clearly, such disparity makes it hard to fully dispatch wind power, even it is generously subsidized to have very low cost. According to \cite{lewis2016wind}, in 2015, the curtailment
(i.e., the undispatched wind power) rate in China is 15\%, and can be as high as 30\% in less
populated provinces.

\begin{figure}[t!]
\begin{center}
\includegraphics[width=0.8\textwidth]{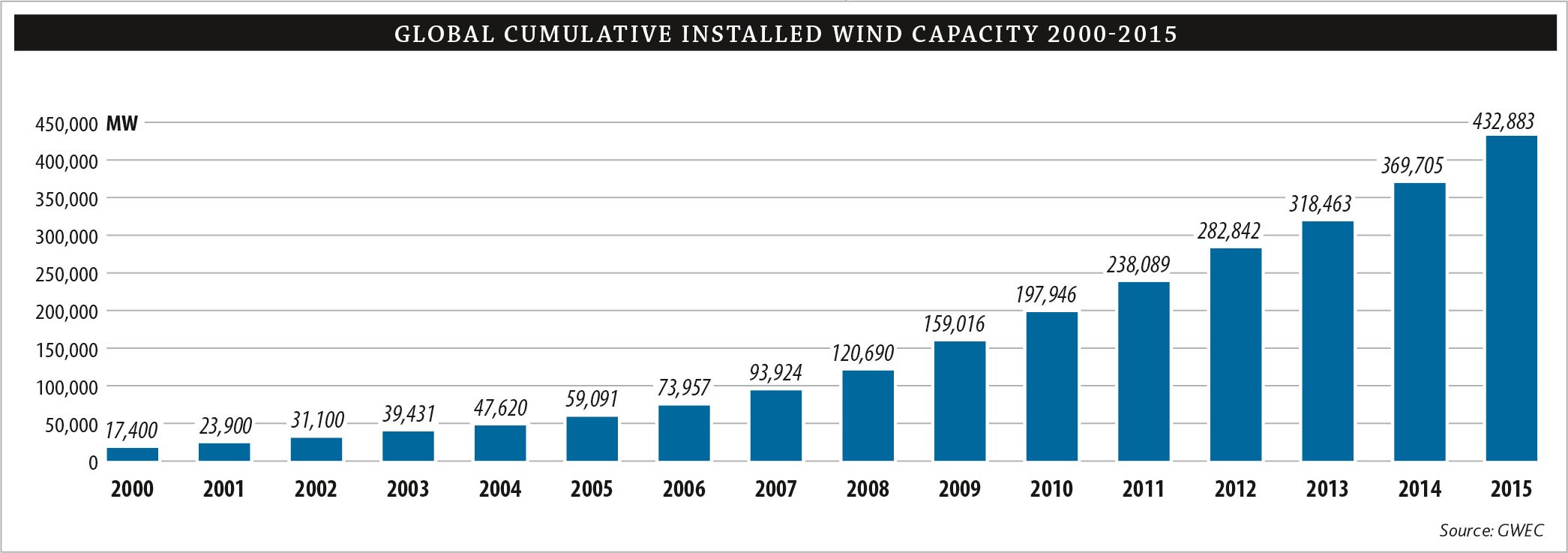}
\end{center}
\caption{Global cumulative installed wind power generation capacity}\label{fig:fig1}
\end{figure}

Obviously, on the one hand, sites and sizes of wind farms should be carefully selected.
An inappropriate location or wrongly determined capacity definitely cannot ensure its wind power dispatchability, which fails to reach the wind power penetration target or recover the investment cost. Hence, an analytical method that considers all critical factors, such as loads, wind resources, existing generation assets and grid configurations, should be employed to support those decisions. This is particularly important for power grids that are operated as electricity markets where market rules
should be followed to clear load and supply \cite{Baringo2011Wind,Baringo2012Transmission,maurovich2015transmission}.
On the other hand, power grid needs to improve its delivery capacity to integrate more wind power. Certainly, building new and stronger transmission and distribution systems, especially long-distance transmission lines, could dispatch more wind power and reduce its curtailment. However, such new projects, which are often across multiple states or provinces, must go through stringent environmental evaluations and approvals, could be
extremely expensive, and require many years to complete. Different from that
idea,  a recent strategy
is to switch off some existing lines to accommodate different
load situations \cite{o2005dispatchable,hedman2011review}. Research has shown that this strategy,
which is often referred to as topology control, can significantly improve the
power delivery capacity and demonstrate a desired cost-effectiveness property \cite{fisher2008optimal,Hedman2008Optimal}.

 Although many recent papers have studied using topology control to achieve higher penetration level of wind and other renewable energy   (e.g., \cite{qiu2015chance,7742999,nikoobakht2016managing,korad2016enhancement,7541026}), its impact on the sitting and sizing of wind farms have not be fully investigated. It basically remains unknown how to analytically determine a wind power generation capacity expansion plan when the electricity market is operated with topology control. Hence, an effective planning tool considering those factors is definitely needed, to support wind power and other renewable energy to reach a high
penetration level and to achieve sustainable development in modern power systems.

\begin{figure}[t!]
    \centering
    \begin{subfigure}{0.45\textwidth}
        \centering
        \includegraphics[height=1.6in]{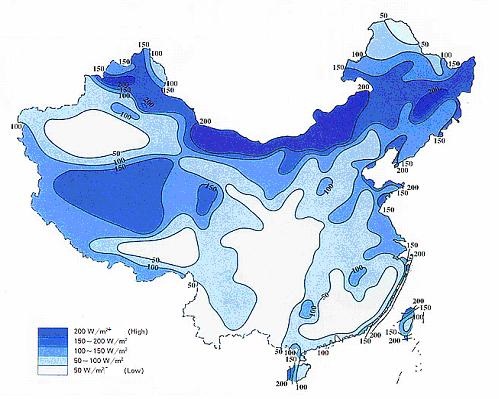}\label{figWR}
        \caption{Wind power resources in China}
    \end{subfigure}%
    \hspace{-0.2cm}
    \begin{subfigure}{0.45\textwidth}
        \centering
        \includegraphics[height=1.6in]{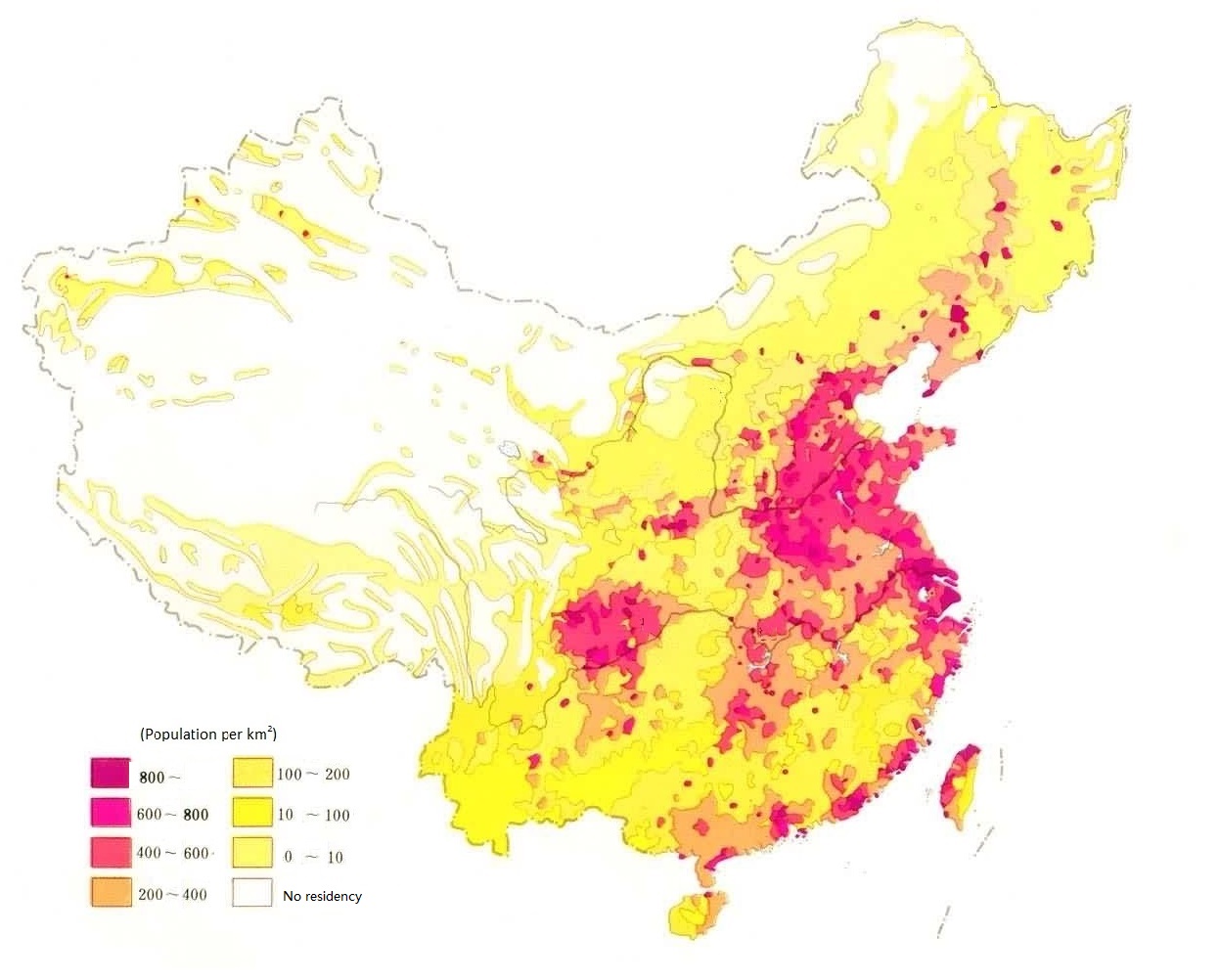}  \label{PD}
        \caption{Population density of China}
    \end{subfigure}
    \caption{Wind resource and population distribution in China} \label{fig:wr_pd}
\end{figure}

%\begin{figure}[htbp]
%\centering
%\begin{minipage}[t]{0.42\textwidth}
%\centering
%\includegraphics[width=\textwidth]{china_re.jpg}\label{fig1}
%\caption{Wind power distribution of China}
%\end{minipage}
%\begin{minipage}[t]{0.5\textwidth}
%\centering
%\includegraphics[width=\textwidth]{Pop_Density_China.jpg}  \label{fig2}
%\caption{Population density of China}
%\end{minipage}
%\end{figure}

 %and propose a bi-level optimization to
% derive demonstrate the role of wind power in an opening electricity market environment. Each wind power investors offers its production at zero price and is paid as the local marginal price (LMP)\cite{Baringo2011Wind} or a reasonable compensation for the elimination of carbon emission\cite{european2009economics}. The upper level objective is to optimize the profit of wind power investment to satisfy the government and investors under given budgets and limited capacity. In the meanwhile, the social welfare (generation cost of thermal power) also need to consider which is formulated as the objective of lower level \cite{Baringo2011Wind,garci2006electricity}. We also consider the topology control on the transmission network as can reduce the cost of thermal power generation.

\subsection{Literature Review}
Over the last twenty years, to address the planning challenges of integrating wind and other renewable energy into power gird, many advanced decision support models and sophisticated computing methods have been developed and studied (e.g., \cite{burke2010maximizing,burke2011study,hamidi2011value,zhan2016generation,kuznia2013stochastic} and related references in \cite{banos2011optimization,connolly2010review}).
Nevertheless, limited research has been extended to consider  the electricity market impact on wind power generation planning, where the absorption of wind power is determined by the market clearing outcomes.

% Among them, most planning research for wind power generation is for vertically-integrated systems in regulated environment where a utility company or controller owns or manages the entire flow of electricity from generation assets to end-users.

To capture the interaction between system planner and electricity market,  bilevel optimization (and its variants) is often adopted where the lower level problem, a linear program (LP) in general, is to derive market clearing results. Many system expansion planning applications can be found in the literature, including those for transmission
e.g.,\cite{garces2009bilevel,hesamzadeh2011bi},  for conventional generation
e.g.,\cite{kazempour2011strategic,kazempour2012strategic,wogrin2011generation} and for joint development of transmission and generation e.g.,\cite{jenabi2013bi,jin2011capacity}. Research presented in \cite{Baringo2011Wind} is probably the first one to explicitly study wind generation planning in an electricity market. On top of the lower level market problem, the upper level determines the wind farm location and capacity decisions to maximize the revenue from  paying absorbed wind power at locational marginal prices. Study presented in \cite{Baringo2012Transmission} extends that in \cite{Baringo2011Wind} by considering joint investment on wind generation and transmission lines.
To handle the computational challenge when a large number of scenarios are introduced to capture load and wind power uncertainties for bilevel model of \cite{Baringo2011Wind}, a Benders decomposition method is developed in \cite{baringo2012wind} to make use of decomposable structures and to efficiently solve typical instances. A recent paper \cite{pineda2016impact} investigates
the forecast errors of uncertain generation units (e.g., wind farms) by considering two markets, i.e., the regular day-ahead market based on nominal forecast and the new real-time balancing market to handle imbalances from forecast errors, and builds a scenario-based bilevel optimization model to simulate the day-ahead market in different stochastic scenarios.  We mention that the typical approach to compute those bilevel mixed integer programs (MIPs) is to reformulate into single level problems by replacing the lower level LPs with the Karush--Kuhn--Tucker (KKT) conditions, and then to linearize them into linear MIPs that are readily computable by a professional MIP solver or by Benders decomposition \cite{baringo2012wind}.

As noted earlier, topology control is a rather new strategy for power systems, which was introduced in \cite{o2005dispatchable} in 2005. Although it does not involve any costly hardware installation or upgrade, according to \cite{fisher2008optimal}, a great transmission capacity increase can be achieved, which leads to a savings of 25\% in system dispatch cost.  As shown in \cite{hedman2010co}, it can be further dynamically co-optimized along with unit commitment decisions to generate
significant economic and reliability benefits in dealing with network contingencies. Hence, many recent papers study integrating topology control into system operations  to better absorb wind and other renewable energy, (e.g., \cite{qiu2015chance,7742999,nikoobakht2016managing,korad2016enhancement,7541026}). Nevertheless, as noted in  \cite{Hedman2008Optimal,hedman2011optimal,han2016impacts}, topology control, as a non-traditional operation, raises new challenges to system operators and electricity markets. Now, fast and reliable topology control algorithms for large-scale grid are still in the development stage, coordination of topology control and other operations are under investigation, and, in particular, many economic and policy implications should be addressed in market environments.

\subsection{Research Objectives and Paper Structure}

 In this paper, to analytically support the development of wind power and other renewable energy, we study the capacity expansion planning problem, i.e., to address the sitting and sizing
 issue of wind farms,  under an electricity market environment that employs topology control to achieve better transmission capacity for less market clearing cost.
 Specifically, we first develop a novel mixed integer bilevel optimization model to represent investment, operation and topology control  decisions constrained within a market environment. To the best of our knowledge, we have not been aware of any existing study on topology control in  market-based bilevel (either planning or operation) models. Because the lower level is an MIP, due to the topology control decisions, it cannot be equivalently replaced by any KKT conditions. To solve this challenging bilevel MIP, we then customize a recent decomposition method and enhance it by making use of grid structure. Finally, by solving  instances obtained from typical IEEE test beds, we demonstrate the effectiveness of our solution method on solving bilevel capacity expansion model, and analyze  the impact of topology control on wind power generation planning and wind power integration.

 Overall, we believe that the presented research will provide novel tools to analytically support wind and other renewable energy development, to gain a deeper understanding on topology control, and to promote the integration of renewable energy in modern electricity markets.

  The remaining of this paper is organized as follows. In Section 2, we present our bilevel MIP model for market-based wind power generation capacity expansion planning problem. In Section 3, a customized computational method using the reformulation and decomposition strategy, along with its enhancements, is described. In Section 4, we provide numerical results and our analysis on typical IEEE test beds. Section 5 concludes this paper with a discussion on the future research directions.

\section{Wind Power Generation Planning Model}
\subsection{Modeling Preparation}
In this study, we consider a single planner that makes sitting and sizing decisions of wind generation facilities, i.e., wind farms, over the power grid that runs an electricity market to clear load and supply. By following the convention of bilevel capacity expansion planning research, we build a static model of wind power generation planning for a single representative year. Hence, all wind farm sitting and sizing decisions are made in the same stage, which could provide a desired tradeoff between modeling accuracy and computational trackability \cite{Baringo2011Wind}. Certainly,  it can be extended to a multi-stage model with timing decisions, which should have an extremely higher computational complexity.

To represent the electricity demand over that single year, we adopt the popular load-duration curve, as in the top level picture in Figure \ref{fig:LW_DC}, that plots demand levels (with their durations) from highest to lowest. Because it captures the long-term demand information in a compact way, it has been often used in various power system planning research, e.g., \cite{shu2015spatial,roh2009market}. To formulate a trackable model,
we follow the discretization strategy \cite{roh2009market,Baringo2011Wind} to
cluster the whole curve into $|T|$  demand blocks ($|T|=4$ in Figure \ref{fig:LW_DC}) and
use their mean values (displayed as bold lines) to represent those curve segments. The lower level picture plots the corresponding wind intensity curves in those blocks, which, again, will be represented by their mean values (displayed as in bold lines). Note also that the fuel-based power generation typically has a nonlinear increasing cost function.  That is, the marginal cost is getting larger for a higher level generation. To handle such nonlinearity issue, we also partition the generation capacity into a few production blocks, and assign different unit costs to those blocks to match that increasing marginal cost trend.

\begin{figure}[t!bp]%[t!]
\begin{center}
\includegraphics[width=0.5\textwidth]{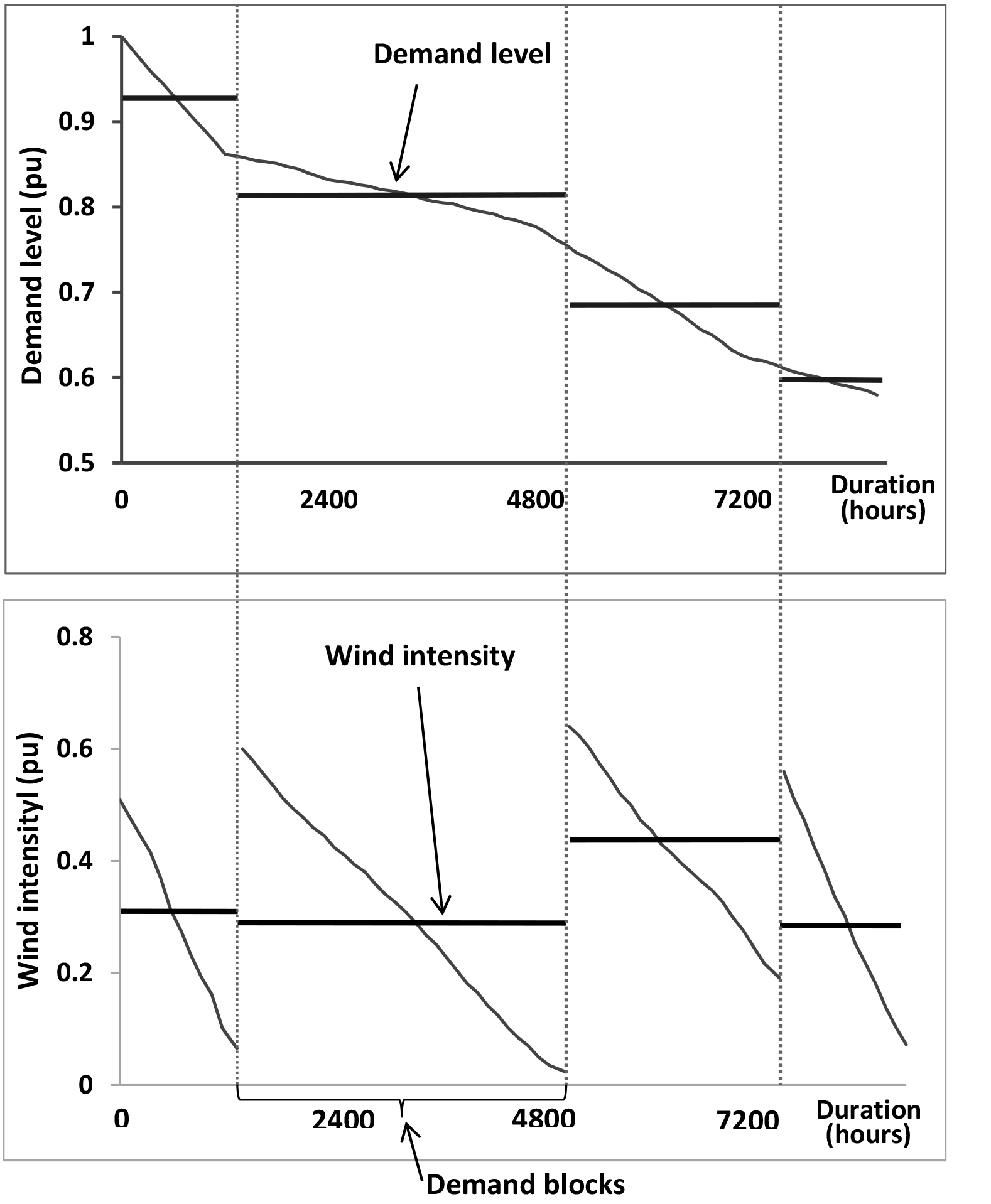}
\end{center}
\caption{Load and wind duration curves}\label{fig:LW_DC}
\end{figure}

Similarly, to capture the
We note that our data preparation is to build a rather deterministic bilevel model, while more samples (i.e., scenarios) from each demand block, as in \cite{Baringo2011Wind}, can be employed to develop a stochastic one. Nevertheless, a stochastic bilevel  MIP model (with lower level MIP problems) imposes a prohibitive computational challenge and should demand for more sophisticated algorithms. Hence, we currently focus on the investigation of topology control and wind power generation planning in a deterministic environment, and leave the study in a stochastic environment as a future research direction.

\subsection{Bilevel MIP Formulation and Description}

%  In this model, the objective in \eqref{eq_otep} of the transmission planner is to minimize the annualized cost of transmission line investment and penalties for load shedding.

%\subsectio
In our bilevel optimization model, the upper level, i.e., the system planner, seeks
optimal sites and sizes of new wind farms to maximize the (weighted) wind power absorption and minimize the (annualized) investment cost in that representative year, subject to budget and wind resource availability restrictions. Note that in a power grid, a site is referred to a bus in the network. Wind power absorption at each bus is actually determined by the market clearing problem, which is the lower level problem that minimizes the dispatch cost with consideration of grid configurations. As mentioned in Section \ref{intro}, we particularly model topology control operations within the lower level clearing problem. We next provide the table of nomenclature that includes parameters and variables, and then present our bilevel wind power generation planning model. When it is clear within the context,  a letter in bold face is to represent a vector of variables denoted by that letter.

\begin{table}[!ht]
\begin{flushleft}
\scalebox{0.8}{
\begin{tabular}{ll}
\hline
\Large{\textbf{Nomenclature}}\\\\
\emph{Indices and Sets} &\\
$\mathbf{T}$ & the set of demand blocks, indexed by $t$\\
$\mathbf{I}$  &  the set of buses, indexed by $i$\\
$\mathbf{L}$ & the set of lines, indexed by $l$\\
$\mathbf{\Omega}$  & the set of fuel-based generators, indexed by $j$\\
$\Omega_i \subseteq \Omega$ & the set of fuel-based generators at bus $i$,\\
$\Psi$  & the set of buses eligible for building wind farms, $\Psi \subseteq \mathbf{I}$\\
$\mathbf{B}_{j}$  & the set of production blocks of generator  $j$, $j \in \mathbf{\Omega}$, indexed by $b$\\
$o(l)$  & the origin bus of transmission line $l$\\
$d(l)$  & the destination bus of transmission line $l$\\
$r\in \Omega$  &  the reference bus\\ \\
\emph{Parameters}&\\
$\kappa$    &  weight coefficient for wind power absorption \\
$L_t$ &  Duration (in hours) in demand block $t$\\
$S_{l}$  & susceptance of line $l$\\
%$\rho$   & annual interest of wind power investment \\
$h_i$ and ${c}_{i}$ & annualized fixed and  investment costs of unit wind power capacity at $i$\\
$H_i$ and $C_i$ & fixed and investment costs of unit wind power generation capacity at $i$\\
$\rho_i$ & load shedding penalty cost that is strictly positive \\
$\hat C$  &  overall budget of wind power investment\\
$\overline U_{i}$  & capacity upper bound of wind farm installation at $i$\\
%$s_{ib}$  & generation size at block $b_{i}$ of bus \emph{i}\\
$P_{jb}$ &  Generation capacity of $b$-th block by fuel-based generator $j$\\
$D_{i,t}$  & power demand at bus $i$ in demand block $t$\\
$k_{i,t}$  & wind intensity at bus $i$ in demand block $t$\\
$\overline f_l$  & transmission capacity of line $l$\\
$\overline \theta$ & maximum value of phase angle  \\
$p_{jb}$  & price offered by generator $j$ in its $b$-th block, $j \in \Omega$\\
\\
\emph{Variables}&\\
$x_{i}$  & binary variables, 1 if a wind farm is installed at bus $i$\\
$u_{i}$  & wind power capacity\\
$g^{w}_{i,t}$  & wind power production at bus $i$ in demand block $t$\\
$g^m_{jb,t}$  & power generation in  $b$-th block by fuel-based generator $j$ in  demand block $t$\\
$s_{i,t}$ & load shedding at bus $i$ in demand block $t$\\
$z_{l,t}$  & binary variables with 0 representing that line $l$ is switched off in demand block $t$ \\
$f_{l,t}$  & power flow on transmission line $l$ in demand block $t$ \\
$\theta_{i,t}$  & phase angle at bus $i$ in demand block $t$\\
\hline
\end{tabular}
}
\end{flushleft}
\end{table}
\newpage

\begin{subequations}
\label{eq_BiMIP}
\begin{align}
\max \ & \kappa\sum_{t\in \mathbf{T}} L_t \sum_{i \in \Psi} g^{w}_{i,t}
 - \sum_{i \in \Psi} (c_iu_i+h_ix_i)  \label{eq_obj_UP}\\ %
 {s.t.} \ &  \sum_{i \in \Psi}  (C_iu_i +H_ix_i) \leq  \hat C \label{eq_budget_UP} \\ %
 & u_i  \leq \overline U_{i} x_i, \ \forall i \in \Psi    \label{eq_installation_UP}\\ %
  & x_{i} \in \{0,1\}, u_{i} \in  \mathbb{{R}_{+}}, i \in  \Psi \label{eq_var_UP}
 \end{align}
 \end{subequations}
where $ (\mathbf{g}^m_t, \mathbf{g}^w_t, \mathbf{\theta}_t, \mathbf{f}_t,\mathbf{s}_t, \mathbf{z}_t) \in$
\begin{subequations}
\label{eq_LowMIP}
\begin{align}
\arg\min \Big\{ \ & \sum_{j \in  \Omega} \sum_{b \in  \mathbf{B}_j} p_{jb}g^m_{jb,t}+ \sum_{i\in \mathbf{I}} \rho_is_{i,t} \label{eq_obj_LW}\\
 {s.t.} \ &  0\leq g^w_{i,t} \leq k_{i,t} u_i , \ \forall i \in \Psi \label{eq_wind_LW}\\
  & 0\leq g^m_{jb,t} \leq P_{jb}, \ \forall j \in \Omega, b \in \mathbf{B}_{j} \label{eq_genUB_LW}\\
 & \sum_{j\in \Omega_i}\sum_{b \in \mathbf{B}_{j}} g^m_{jb,t} + g^{w}_{i,t} + \sum_{l|d(l)=i} f_{l,t}
  + s_{i,t} = D_{i,t} +  \sum_{l|o(l)=i} f_{l,t}, \ \forall i\in \Psi \label{eq_flow1_LW}\\
 & \sum_{j\in \Omega_i}\sum_{b \in \mathbf{B}_{j}} g^m_{jb,t} +
 \sum_{l|d(l)=i} f_{l,t} + s_{i,t} = D_{i,t} +  \sum_{l|o(l)=i} f_{l,t}, \ \forall i\in \mathbf{I}/\Psi \label{eq_flow2_LW}\\
 & f_{l,t} = S_{l} \big(\theta_{(o(l),t)}-\theta_{(d(l),t)}\big ) z_{l,t}, \ \forall l\in \mathrm{L} \label{eq_flowline_LW}\\
 & -\overline f_l \leq f_{l,t} \leq \overline f_{l}, \ \forall l \in \mathbf{L} \label{eq_flowB_LW}\\
    & -\overline \theta \leq \theta_{i,t} \leq \overline \theta, \ \forall i \in \mathbf{I}/\{r\} \label{eq_angleB_LW}\\
 &     z_{l,t} \in \{0,1\}, \forall l\in \mathrm{L}, s_{i,t}\geq 0, \forall i \in \mathbf{I} \label{eq_var_LW} \Big\}
\end{align}
\end{subequations}
 for all $t\in \mathbf{T}$.

For the upper level problem, the objective function in \eqref{eq_obj_UP} is to maximize the (weighted) annual wind power absorption and the annualized fixed and investment costs. Constraint in \eqref{eq_budget_UP}
defines the total budget to restrict installed wind farms and their capacities. Constraints \eqref{eq_installation_UP}  provide upper bounds on wind farm capacities.

The lower level problem in \eqref{eq_LowMIP} provides market clearing results using DC optimal power flow formulation. We highlight that a set of \eqref{eq_LowMIP} will be defined for all demand block in $T$. Unless explicitly mention, we treat the aggregation of those problems together as the lower level problem in our bilevel optimization model. Specifically, the objective function \eqref{eq_obj_LW} is to maximize social welfare, which is translated to minimize the total generation cost from fuel-based power generators over their production blocks. Note that wind power is assumed to
 have zero cost. Constraints in \eqref{eq_wind_LW} impose upper bound restrictions on the wind power injection due to available wind intensity, for buses with wind farms.  Constraints in \eqref{eq_genUB_LW} ensure upper bound restrictions on power generation of each fuel-based generator and its every production block.  Constraints in \eqref{eq_flow1_LW} and \eqref{eq_flow2_LW} present power balance requirements for buses with and without wind farms. Constraints in \eqref{eq_flowline_LW} define, if that line is not disconnected, i.e., $z_{l,t}=1$, the power flow is proportion by $S_l$ to the phase angle difference between two ending buses of each line. Otherwise, the power flow is $0$.
 Constraints in \eqref{eq_flowB_LW} provide lower and upper bounds for power flow in each line.
 Constraints in \eqref{eq_angleB_LW} bound phase angle for each bus, except the reference bus, whose angle is fixed at 0.

  Obviously, the upper level problem in \eqref{eq_BiMIP} is a typical budget allocation formulation. When line switching operations are not introduced, the lower level problem in \eqref{eq_LowMIP} reduces to a pure linear program, which is the case in the majority of exiting bilevel optimization based power system capacity expansion models, e.g., \cite{Baringo2011Wind,baringo2012wind}. Nevertheless, different from them, we highlight that \eqref{eq_LowMIP} is a mixed integer nonlinear program, given that $z_{l,t}$'s are binary and \eqref{eq_flowline_LW} are nonlinear. With the binary value of $z_{l,t}$ and upper bound of $\theta_{i,t}$, the nonlinear terms in \eqref{eq_flowline_LW} can be linearized by using additional variables and constraints. For example, consider $\theta_{(o(l),t)}z_{l,t}$ in \eqref{eq_flowline_LW}. It can be equivalently replaced by a new variable $\vartheta_{(o(l),t)}$ and the associated constraints:
  \begin{subequations}
  \begin{align}
  -\overline \theta z_{l,t}\leq &\vartheta_{(o(l),t)}  \leq \overline \theta z_{l,t}, \\
  \theta_{(o(l),t)}-\overline\theta (1-z_{l,t})\leq & \vartheta_{(o(l),t)} \leq
  \theta_{(o(l),t)}+\overline\theta (1-z_{l,t})
  \end{align}
  \end{subequations}
  By using such linearization technique to all nonlinear terms in \eqref{eq_flowline_LW}, the lower level problem \eqref{eq_LowMIP} can be readily converted into an MIP. Since those linearization operations are rather simple, we
  do not present the complete linearized model and still use (\ref{eq_BiMIP}-\ref{eq_LowMIP}) with \eqref{eq_flowline_LW} as our bilevel planning model to gain intuitive  understanding. \\

\noindent\textbf{Remark:} Although the complicated nonlinear factor can be addressed  by linearization, the essential challenge to solve the bilevel optimization model (\ref{eq_BiMIP}-\ref{eq_LowMIP}) comes from the mixed integer formulation of \eqref{eq_LowMIP}. Because of the non-convexity nature from discrete variables, the popular KKT conditions based solution approach is not applicable anymore. Indeed, algorithm development for bilevel mixed integer program with a mixed integer lower level problem is rather limited and few real applications have been accurately solved. Actually, the existence of integer variables could
cause the whole bilevel MIP model not have any optimal solution \cite{Zeng2014Solving}.
For instances with an optimal solution, a recent algorithm strategy, i.e., reformulation and decomposition method \cite{Zeng2014Solving}, has demonstrated a strong solution capability in computing bilevel mixed integer capacity expansion problems \cite{campoacapacity,CMU_bilevel}. Hence, in the next section, we adopt this method and present a customization according to specifications of our bilevel wind power planning MIP model. Moreover, we design a few enhancement techniques that could further significantly improve our computational power over practical instances. Before we proceed to the next section, we state a result on the existence of an optimal solution to (\ref{eq_BiMIP}-\ref{eq_LowMIP}). Let $\mathbf{x}$ and $\mathbf{u}$ denote the vectors of $x_i$ and $u_i$, $i\in\Psi$, respectively.

\begin{prop}
$(i)$ The lower level problem in \eqref{eq_LowMIP} has a finite optimal solution
for any combination of $(\mathbf{x},\mathbf{u})$ and $\mathbf{z}_t$, i.e., it has the \emph{relatively complete response} property.\\
$(ii)$ The complete bilevel MIP model in (\ref{eq_BiMIP}-\ref{eq_LowMIP}) has an optimal solution.
\end{prop}
The first result follows from the fact that the load shedding penalties guarantee the existence of an optimal solution to \eqref{eq_LowMIP} for any given $(\mathbf{x},\mathbf{u})$ and $\mathbf{z}_t$. Then, according to Corollary 3 in \cite{Zeng2014Solving}, the second result is valid.

\section{Solution Method Based on Reformulation and Decomposition }
As mentioned earlier, the  straightforward KKT conditions based solution method cannot be applied due to the existence of binary variables for line switching operations in the lower level problem. As shown in \cite{Zeng2014Solving}, the reformulation and decomposition method can effectively address that challenge, which is customized and enhanced in this section. Because the lower level problem \eqref{eq_LowMIP} is complicated, we next provide a compact matrix-based representation to make our exposition in this section more accessible.

%\begin{subequations}
%\label{eq_simplifyUP}
% \begin{align}
% \max & \kappa \sum_{t\in \mathbf{T}} L_t\mathbf{1}^T\mathbf{g}^w_t-(\mathbf{cu}+\mathbf{hx})\\
% \mbox{s.t.} \ & \mathbf{Cu}+\mathbf{Hx}\leq \hat C\\
%  &  \mathbf{u}\leq \mathbf{Vx}\\
%  &  \mathbf{x}\in \{0,1\}^{|\Phi|}, \mathbf{u}\geq \mathbf{0}.
% \end{align}
% \end{subequations}
% with
\begin{subequations}
\label{eq_simplify}
 \begin{align}
(\mathbf{g}^m_t, \mathbf{g}^w_t, \mathbf{\theta}_t, \mathbf{f}_t,\mathbf{s}_t, \mathbf{z}_t) \in \arg & \Big\{\min \mathbf{p}^m\mathbf{g}^m_t+\mathbf{\rho}\mathbf{s}_t: \\
\mbox{s.t.} %\ & \mathbf{g}^w_t\leq \mathbf{Ku}\\
\ & A^m\mathbf{g}^m_t+A^w\mathbf{g}^w_t+A^f\mathbf{f}_t+ A^s\mathbf{s}_t=\mathbf{D}_t \\
& f_{l,t}+\theta^T_tJ\mathbf{z}_t = 0, \ \forall l\in \mathbf{L}\label{eq_comp_flowline}\\
& B^m\mathbf{g}^m_t+B^w\mathbf{g}^w_t+B^f\mathbf{f}_t+B^{a}\mathbf{\theta}_t \geq \mathbf{b} + \mathbf{Ku} \label{eq_comp_bound}\\
& \mathbf{g}^m_t\geq \mathbf{0}, \mathbf{g}^w_t\geq \mathbf{0}, \mathbf{s}_t\geq 0, \mathbf{z}_t\in \{0,1\}^{|\mathbf{L}|}
\Big\}
\end{align}
\end{subequations}
where, with properly defined coefficient matrices, the first constraint represents those in (\ref{eq_flow1_LW}-\ref{eq_flow2_LW}), the second constraint represents those in \eqref{eq_flowline_LW} and the third constraint include all other constraints, i.e., bound constraints. We mention that superscript $T$ in \eqref{eq_comp_flowline}
denotes the transpose operation and vector $\mathbf{u}$ represents the capacity variables from the upper level problem.

In the following, we describe the basic idea and concrete steps of our computational method using the compact formulation of \eqref{eq_LowMIP}.

\subsection{An Equivalent Reformulation for Decomposition}
To provide a decomposable structure for algorithm development, we follow \cite{Zeng2014Solving} to reformulate our original bilevel MIP model in (\ref{eq_BiMIP}-\ref{eq_LowMIP}) as follows. Note that $\mathbf{1} \in \mathbb{R}^{|\Psi|}$ is a vector of $1$s  $\mathbf{c}$, $\mathbf{h}$
and $\mathbf{x}$ are vectors of $c_i$, $h_i$ and $x_i$ respectively, and other notations are introduced in the compact form in \eqref{eq_simplify}.

\begin{subequations}\label{eq_equivalentBMIP}
\begin{align}
\max \ & \kappa\sum_{t\in \mathbf{T}} L_t (\mathbf{1}^T\mathbf{g}^{w}_t) -
  (\mathbf{cu}+ \mathbf{hx})  \label{eq_obj_MP}\\ %
  {s.t.} \ &  (\ref{eq_budget_UP} - \ref{eq_var_UP}) \label{M2} \\ %
 \ & A^m\mathbf{\tilde g}^m_t+A^w\mathbf{\tilde g}^w_t+A^f\mathbf{\tilde f}_t+ A^s\mathbf{\tilde s}_t=\mathbf{D}_t, \forall t\in \mathbf{T}\\
& \tilde f_{l,t}+\mathbf{\tilde \theta}^T_tJ\mathbf{\tilde z}_t = 0, \ \forall l\in \mathbf{L}, t\in \mathbf{T} \\
& B^m\mathbf{\tilde g}^m_t+B^w\mathbf{\tilde g}^w_t+B^f\mathbf{\tilde f}_t+B^{a}\mathbf{\tilde \theta}_t \geq \mathbf{b} + \mathbf{Ku}, \forall t\in \mathbf{T}  \\
& \mathbf{\tilde g}^m_t\geq \mathbf{0}, \mathbf{\tilde g}^w_t\geq \mathbf{0}, \mathbf{\tilde s}_t\geq 0, \mathbf{\tilde z}_t\in \{0,1\}^{|\mathbf{L}|}, \forall t\in \mathbf{T} \label{eq_tilde_end}\\
& \mathbf{p}^m\mathbf{\tilde g}^m_t+\mathbf{\rho}\mathbf{\tilde s}_t \leq \min\Big\{\mathbf{p}^m\mathbf{g}^m_t+\mathbf{\rho}\mathbf{s}_t: \label{eq_refor_link}\\
& \ \  \ \ \ \ \ \  \ \ \ \ \ \mbox{s.t.} %\ & \mathbf{g}^w_t\leq \mathbf{Ku}\\
\ A^m\mathbf{g}^m_t+A^w\mathbf{g}^w_t+A^f\mathbf{f}_t+ A^s\mathbf{s}_t=\mathbf{D}_t \\
&  \ \  \ \ \ \ \ \  \ \ \ \ \ \ \ f_{l,t}+\theta^T_tJ\mathbf{z}_t = 0, \ \forall l\in \mathbf{L} \\
&  \ \  \ \ \ \ \ \  \ \ \ \ \  \ \ B^m\mathbf{g}^m_t+B^w\mathbf{g}^w_t+B^f\mathbf{f}_t+B^{a}\mathbf{\theta}_t \geq \mathbf{b} + \mathbf{Ku}  \\
&  \ \  \ \ \ \ \ \  \ \ \ \ \ \ \ \mathbf{g}^m_t\geq \mathbf{0}, \mathbf{g}^w_t\geq \mathbf{0}, \mathbf{s}_t\geq 0, \mathbf{z}_t\in \{0,1\}^{|\mathbf{L}|}\Big\}
& \forall t\in \mathbf{T} \label{eq_refor_link_end}
\end{align}
\end{subequations}

Because of \eqref{eq_refor_link}, we can conclude that \eqref{eq_equivalentBMIP} is  equivalent to the original bilevel MIP model in (\ref{eq_BiMIP}-\ref{eq_LowMIP}). Although  more complicated than
(\ref{eq_BiMIP}-\ref{eq_LowMIP}), it, however, provides a convenient representation to derive non-trivial bounds to (\ref{eq_BiMIP}-\ref{eq_LowMIP}). Specifically, let $\mathbf{Z}_t$ be the collection of all possible realizations of $\mathbf{z}_t$. Clearly, $|\mathbf{Z}_t|= 2^{|\mathbf{L}|}$.
Next, (\ref{eq_refor_link}-\ref{eq_refor_link_end}) can be rewritten by enumerating $\mathbf{z}_t$ as
\begin{subequations}\label{eq_rewritten}
\begin{align}
& \mathbf{p}^m\mathbf{\tilde g}^m_t+\mathbf{\rho}\mathbf{\tilde s}_t \leq \min \Big\{\mathbf{p}^m\mathbf{g}^{m,q}_t+\mathbf{\rho}\mathbf{s}^q_t: \label{eq_rewritten_link}\\
& \ \  \ \ \ \ \ \  \ \ \ \ \ \mbox{s.t.} %\ & \mathbf{g}^w_t\leq \mathbf{Ku}\\
\ A^m\mathbf{g}^{m,q}_t+A^w\mathbf{g}^{w,q}_t+A^f\mathbf{f}^q_t+ A^s\mathbf{s}^q_t=\mathbf{D}_t \label{eq_rewritten_link1}\\
&  \ \  \ \ \ \ \ \  \ \ \ \ \ \ \ f^q_{l,t}+{(\theta^q)}^T_tJ\mathbf{z}^{q*}_t = 0, \ \forall l\in \mathbf{L} \label{eq_rewritten_link2}\\
&  \ \  \ \ \ \ \ \  \ \ \ \ \  \ \ B^m\mathbf{g}^{m,q}_t+B^w\mathbf{g}^{w,q}_t+B^f\mathbf{f}^q_t+B^{a}\mathbf{\theta}^q_t \geq \mathbf{b} + \mathbf{Ku}  \label{eq_rewritten_link3}\\
&  \ \  \ \ \ \ \ \  \ \ \ \ \ \ \ \mathbf{g}^{m,q}_t\geq \mathbf{0}, \mathbf{g}^{w,q}_t\geq \mathbf{0}, \mathbf{s}^q_t\geq 0\Big\}, \ \forall \mathbf{z}^{q*}_t\in  \mathbf{Z}_t,  \label{eq_rewritten_link_end}
\end{align}
\end{subequations}
where $\mathbf{z}^{q*}_t\in  \mathbf{Z}_t$ is a particular realization of $\mathbf{z}_t$.

Although \eqref{eq_rewritten} is definitely cumbersome due to enumeration, it has two critical advantages. First, once $\mathbf{z}^{q*}_t$ is provided, the right-hand-side of (\ref{eq_rewritten_link}-\ref{eq_rewritten_link_end}) is an LP. Because it has a finite optimal value that can be characterized by its KKT conditions, we have
\begin{subequations}\label{eq_rewrittenKKT}
\begin{align}
& \ \  \ \ \ \ \ \  \ \ \ \ \ \mathbf{p}^m\mathbf{\tilde g}^{m}_t+\mathbf{\rho}\mathbf{\tilde s}_t \leq \mathbf{p}^m\mathbf{g}^{m,q}_t+\mathbf{\rho}\mathbf{s}^q_t \label{eq_rewritten_linkKKT}\\
& \ \  \ \ \ \ \ \  \ \ \ \ \ \ \    (\ref{eq_rewritten_link1}-\ref{eq_rewritten_link3}) \nonumber \\
& \ \  \ \ \ \ \ \  \ \ \ \ \ \ \ A^m \mathbf{\pi}^q_t +B^m \mathbf{\eta}^q_t \leq \mathbf{p}^m   \label{eq_rewritten_linkKKT1} \\
& \ \  \ \ \ \ \ \  \ \ \ \ \ \ \ A^w \mathbf{\pi}^q_t + B^w \mathbf{\eta}^q_t \leq \mathbf{0}  \label{eq_rewritten_linkKKT2} \\
& \ \  \ \ \ \ \ \  \ \ \ \ \ \ \ A^f \mathbf{\pi}^q_t + \sum_{l \in \mathbf{L}} {\lambda}_{l,t}^n + B^f\mathbf{\eta}^q_t = \mathbf{0} \label{eq_rewritten_linkKKT3}\\
& \ \  \ \ \ \ \ \  \ \ \ \ \ \ \ A^s\mathbf{\pi}^q_t \leq \mathbf{\rho} \label{eq_rewritten_linkKKT4}\\
& \ \  \ \ \ \ \ \  \ \ \ \ \ \ \sum_{l \in \mathbf{L}}J\mathbf{z}^{q*}_t{\lambda}_{l,t}^q + B^{a}\eta^q_t = \mathbf{0} \label{eq_rewritten_linkKKT5}\\
& \ \  \ \ \ \ \ \  \ \ \ \ \ \ \ {\eta}_t^q \perp \Big(B^m\mathbf{g}^{m,q}_t+B^w\mathbf{g}^{w,q}_t+B^f\mathbf{f}^q_t+B^{a}\mathbf{\theta}^q_t - \mathbf{b} - \mathbf{Ku}\Big)  \label{MI_Slack1}\\
 & \ \  \ \ \ \ \ \  \ \ \ \ \ \ \    \mathbf{g}^{m,q} \perp \Big(A^m \mathbf{\pi}^q_t +B^m \mathbf{\eta}^q_t - \mathbf{p}^m\Big)  \label{MI_Slack2}  \\
& \ \  \ \ \ \ \ \  \ \ \ \ \ \ \ \mathbf{g}^{w,q}_t \perp \Big(A^w \mathbf{\pi}^q_t + B^w \mathbf{\eta}^q_t\Big)   \label{MI_Slack3} \\
   & \ \  \ \ \ \ \ \  \ \ \ \ \ \ \ s_t^q \perp \Big(A^s\mathbf{\pi}^q_t - \rho\Big) \label{MI_Slack4} \\
& \ \  \ \ \ \ \ \  \ \ \ \ \ \ \ \mathbf{g}^{m,q}_t\geq \mathbf{0}, \mathbf{g}^{w,q}_t\geq \mathbf{0}, \mathbf{s}^q_t\geq 0, \mathbf{\pi}_t^q,  \mathbf{\lambda}_{l,t}^q~free, \mathbf{\eta}_t^q \geq 0 \label{MI_Var}
\}
\end{align}
\end{subequations}
where $\mathbf{\pi}_t^q, \mathbf{\lambda}_{l,t}^q$, and $\mathbf{\eta}_t^q$ are dual variables of constraints in \eqref{eq_rewritten_link1}, \eqref{eq_rewritten_link2} and \eqref{eq_rewritten_link3} respectively.

Second, instead of having a complete enumeration, (\ref{eq_rewritten_link}-\ref{eq_rewritten_link_end}) developed based on a subset $\mathbf{\hat Z}_t \subseteq \mathbf{Z}_t$ leads to a relaxation of \eqref{eq_equivalentBMIP}, or equivalently, a relaxation of  (\ref{eq_BiMIP}-\ref{eq_LowMIP}). Those two critical advantages enable us to develop an  decomposition algorithm using the column-and-constraint generation method~\cite{Zeng2013Solving}. \vspace{5pt}

\noindent\textbf{Remark:} \\
Note that the nonlinear complementarity constraints in (\ref{MI_Slack1})-(\ref{MI_Slack4}) can be linearized. Consider $\hat i^{\textrm{th}}$ constraint in (\ref{MI_Slack2}) as an example. It is equivalent to constraints  in (\ref{Linearization}), where $M$ is a sufficiently large number.

\begin{subequations}\label{Linearization}
\begin{align}
     & \ \ \mathbf{g}^{m,q}_{\hat i} \leq M \delta_{\hat i} \\
     & \ \ \Big(A^m \mathbf{\pi}^q_t +B^m \mathbf{\eta}^q_t - \mathbf{p}^m\Big)_{\hat i} \geq M(\delta_{\hat i}-1)\\
     & \ \ \delta_{\hat i} \in \{0,1\}
\end{align}
 \end{subequations}

\subsection{Decomposition Algorithm}
As a decomposition algorithm, two subproblems and one master problem are involved in an iterative procedure. We first present two subproblems and then introduce the master problem within the algorithm description.

For a given upper level decision $\mathbf{u^*}$, we formulate and compute the following subproblem $\mathbf{SP1}$. Note that it is  defined for every $t\in \mathbf{T}$.  As mentioned earlier, it can be linearized and readily computed as an MIP problem.
\begin{subequations}\label{SP1}
\begin{align}
\mathbf{SP1}: \ \ \Phi_t(\mathbf{u^*})= \min \ & \mathbf{p}^m\mathbf{g}^m_t+\mathbf{\rho}\mathbf{s}_t  \label{SP1_Obj}\\ %
\mbox{s.t.}~ &  A^m\mathbf{g}^m_t+A^w\mathbf{g}^w_t+A^f\mathbf{f}_t+ A^s\mathbf{s}_t=\mathbf{D}_t    \label{SP1_1}      \\ %
 & f_{l,t}+\theta^T_tJ\mathbf{z}_t = 0, \ \forall l\in \mathbf{L} \label{SP1_2}  \\
 & B^m\mathbf{g}^m_t+B^w\mathbf{g}^w_t+B^f\mathbf{f}_t+B^{a}\mathbf{\theta}_t \geq \mathbf{b} + \mathbf{Ku}^*    \label{SP1_3}   \\
  & \mathbf{g}^m_t\geq \mathbf{0}, \mathbf{g}^w_t\geq \mathbf{0}, \mathbf{s}_t\geq 0, \mathbf{z}_t\in \{0,1\}^{|\mathbf{L}|}   \label{SP1_4}
 \end{align}
 \end{subequations}

 Clearly, \textbf{SP1} provides an optimal solution of lower level model (\ref{eq_LowMIP}) under a wind power investment plan $\mathbf{u^*}$. However, it might have multiple optimal solutions. To derive one that is in favor of \eqref{eq_BiMIP}, we define and compute subproblem \textbf{SP2}, which aggregates all demand blocks. Again, \textbf{SP2} can be readily computed as an MIP problem.

\begin{subequations}\label{SP2}
\begin{align}
(\mathbf{SP2})\ \  \overline \Phi (\mathbf{u^*}) = \max \ & \ \kappa\sum_{t\in T} L_t (\mathbf{1}^T\mathbf{g}^{w}_t) -   (\mathbf{cu}+ \mathbf{hx})  \label{SP2_Obj}\\ %
\mbox{s.t.}~ &  A^m\mathbf{g}^m_t+A^w\mathbf{g}^w_t+A^f\mathbf{f}_t+ A^s\mathbf{s}_t=\mathbf{D}_t    \ \forall t\in \mathbf{T} \label{SP2_1}      \\ %
 & f_{l,t}+\theta^T_tJ\mathbf{z}_t = 0, \ \forall l\in \mathbf{L} \ \forall t\in \mathbf{T}  \label{SP2_2}  \\
 & B^m\mathbf{g}^m_t+B^w\mathbf{g}^w_t+B^f\mathbf{f}_t+B^{a}\mathbf{\theta}_t \geq \mathbf{b} + \mathbf{Ku}^* \ \forall t\in \mathbf{T}    \label{SP2_3}   \\
  & \mathbf{p}^m\mathbf{g}^m_t+\mathbf{\rho}\mathbf{s}_t  \leq \Phi_t(\mathbf{u^*}) \ \forall t\in \mathbf{T}  \label{SP2_4} \\
  & \mathbf{g}^m_t\geq \mathbf{0}, \mathbf{g}^w_t\geq \mathbf{0}, \mathbf{s}_t\geq 0, \mathbf{z}_t\in \{0,1\}^{|\mathbf{L}|}  \ \forall t\in \mathbf{T}  \label{SP2_5}
 \end{align}
 \end{subequations}

Next, we provide the concrete steps of our decomposition algorithm. Let $LB$, $UB$ represent lower and upper bounds of the algorithm and $\varepsilon$ represent the given optimality tolerance.

\vspace{5pt}
\noindent\rule{8cm}{0.4pt}

%To draw a horizontal line (or horizontal rule), set the x-slope to 1 and the y-slope to zero.

\noindent \textbf{Step 1}  Set $LB = - \infty$, $UB = +\infty$,
$\mathbf{\hat Z}_t=\emptyset$ for $t\in \mathbf{T}$, and the iteration counter $\mathbf{j} = 1$.\\

\noindent \textbf{Step 2} Solve the master problem (\textbf{MP}).
\begin{subequations}\label{eq_master}
\begin{align}
(\mathbf{MP}) \ \Theta = \max \ & \kappa\sum_{t\in T} L_t (\mathbf{1}^T\mathbf{g}^{w}_t) -
  (\mathbf{cu}+ \mathbf{hx})  \label{eq_obj_MILP}\\ %
  \mbox{s.t.} \ & (\ref{M2} - \ref{eq_tilde_end}) \label{MI_same} \\ %
   & (\ref{eq_rewritten_linkKKT}- \ref{MI_Var}) \ \forall \mathbf{z}^{q*}_t\in \mathbf{\hat Z}_t, \forall t\in \mathbf{T}
 \end{align}
 \end{subequations}
  Derive an optimal solution, report the upper level decisions $\mathbf{x}^*$ and $\mathbf{u}^*$, and update $UB = \min\{\Theta,  UB\}$.\\

\noindent \textbf{Step 3} Solve subproblems $\mathbf{SP1}$ for given $\mathbf{u}^*$ and for  $t\in \mathbf{T}$ and report their optimal values $\Phi_t(\mathbf{u}^*)$. \\

\noindent \textbf{Step 4} Solve subproblem $\mathbf{SP2}$ for given $\mathbf{u}^*$ and $\Phi_t(\mathbf{u}^*)$. Report optimal binary variables $\mathbf{z}^{q*}_t,  t\in \mathbf{T}$   and
the optimal value $\overline \Phi(\mathbf{u^*})$. Update
$LB = \max\{ \overline \Phi(\mathbf{u^*}), LB\}$.\\

\noindent \textbf{Step 5} If $\frac{UB - LB}{|LB|} \leq\varepsilon$, return $LB$, report $(\mathbf{x}^*, \mathbf{u}^*)$ and terminate. Otherwise, goto \textbf{Step 6}.\\

\noindent \textbf{Step 6} Set $\mathbf{j}= \mathbf{j} + 1$, and update
$\mathbf{\hat Z}_t=\mathbf{\hat Z}_t\cup\{\mathbf{z}^{q*}_t\}$  for $t\in \mathbf{T}$. Go to \textbf{Step 2}.
\noindent\rule{8cm}{0.4pt}\\

\vspace{10pt}
Note that the  complementarity constraints in \textbf{MP} can be linearized as in \eqref{Linearization}, which results in an MIP problem. So, all subproblems and master problem can be computed by a professional MIP solver. Following above steps, our algorithm terminates with an optimal wind power generation expansion plan $(\mathbf{x}^*, \mathbf{u}^*)$ in finite iterations \cite{Zeng2014Solving}.

\subsection{Grid Structure Based Algorithm Enhancements}
It can be anticipated that our decomposition algorithm, as an iterative procedure, will have a heavy computational burden to deal with large-scale power grids, especially for the master problem. To alleviate that computational burden, we propose to analyze the underlying network structure,  eliminate unnecessary variables, and generate valid inequalities to narrow down the search space.

Situations under our consideration satisfy these two assumptions: $(i)$ the least load shedding cost, i.e., $\displaystyle\min_{i\in \Psi}\rho_i$, is strictly higher than  the most expensive generation cost, i.e., $\displaystyle\max_{j\in \Omega, b\in B_j}p_{j,b}$; and $(ii)$ even without any wind power generation, i.e., $\displaystyle x_i=0, i\in \Psi$, the lower level problem \eqref{eq_LowMIP} has a market clearing solution without incurring load shedding. The first assumption reflects the actual practice that heavily penalizes any possible forced load shedding. The second assumption simply states that the existing grid is able to use fuel-based generators to satisfy demand, which is valid for most systems that have to hedge against  wind power intermittence and volatility.

Consider bus $i$ and let $\mathbf{L}(i)$ be the set of lines that link to $i$.
\begin{prop}
\label{prop_singlebus}
If bus $i$ has a non-zero demand in demand block $t$, and is neither equipped with fuel-based generation nor eligible for wind farm installation, we have
\begin{eqnarray}
\label{eq_eh_1}
\sum_{l\in \mathbf{L}(i)} z_{l,t}\geq 1
\end{eqnarray}
is valid for any optimal solution to (\ref{eq_BiMIP}-\ref{eq_LowMIP}).
\end{prop}
\begin{proof}
Based on the aforementioned two assumptions, we can easily conclude that, regardless of the upper level decisions on wind farm installations,
bus $i$ must be connected to the rest of grid in any optimal solution to \eqref{eq_LowMIP} defined for that demand block $t$. So, \eqref{eq_eh_1} follows naturally.
\end{proof}
Without hurting any possible optimal solutions, we can include \eqref{eq_eh_1} into the lower level problem \eqref{eq_LowMIP}. Note that when $\mathbf{L}(i)$ is a singleton, say $\{l_1\}$, we can simply eliminate variable $z_{l_1,t}$ from \eqref{eq_LowMIP} as $z_{l_1,t}$ equals to one.

The idea behind inequality \eqref{eq_eh_1} can be generalized from a single bus to a partition of the grid. Consider a partition $(\mathbf{V}, \mathbf{W})$ of $\mathbf{I}$, i.e.,
$\mathbf{V}\cup\mathbf{W}=\mathbf{I}$ and $\mathbf{V}\cap\mathbf{W}=\emptyset$. Let $\mathbf{L}(\mathbf{V}, \mathbf{W})$ be the collection of lines between $\mathbf{V}$ and $\mathbf{W}$.

\begin{prop}
\label{prop_twopartition}
 If the aggregated generation capacity within $\mathbf{V}$ or $\mathbf{W}$, is not sufficient to meet with demand, i.e.,  either \\
  $\sum_{i\in \mathbf{V}} (D_i,t-\sum_{j\in \Omega_i}\sum_{b\in B_j}P_{jb})-\sum_{i\in \Psi\cap\mathbf{V}}\overline U_i>0$, or \\
 or $\sum_{i\in \mathbf{W}} (D_i,t-\sum_{j\in \Omega_i}\sum_{b\in B_j}P_{jb})-\sum_{i\in \Psi\cap\mathbf{W}}\overline U_i>0,$,  we have
\begin{eqnarray}
\label{eq_eh_2}
\sum_{l\in \mathbf{L}(\mathbf{V}, \mathbf{W})} z_{l,t}\geq 1.
\end{eqnarray}
is valid for any optimal solution to (\ref{eq_BiMIP}-\ref{eq_LowMIP}).
\end{prop}
Again, \eqref{eq_eh_2} can be included into \eqref{eq_LowMIP} defined for that demand block $t$

\begin{figure}[!htbp]
\begin{center}
\includegraphics[scale=0.4]{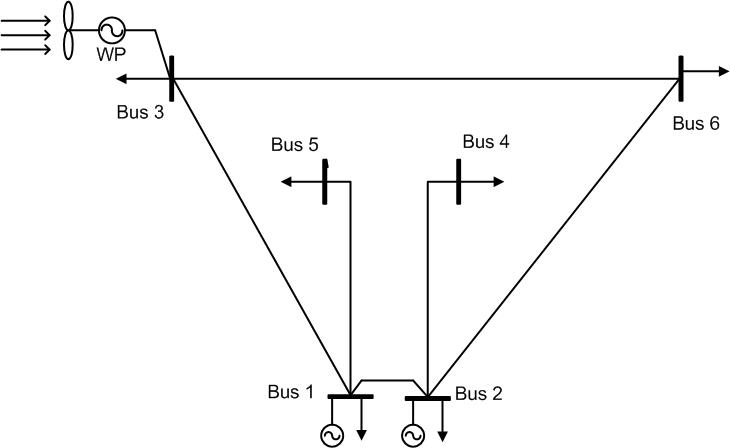}
\end{center}
\caption{6-Bus power system}\label{fig:6bus}
\end{figure}

Next, we use a 6-bus system presented in Figure
\ref{fig:6bus} to illustrate those valid inequalities. We assume that all lines have  sufficient capacities and ignore the phase angle issue to simplify our illustration.
Suppose that we have a single demand block (hence we drop subscript $t$) and demands on those six buses are $5, 10, 0 , 5, 15, 10$ MW, respectively. Buses 1 and 2 have two fuel-based generators that have 30 and 20 MW generation capacities, respectively. Bus 3 is a node that is eligible to install a wind farm with up to 10 MW capacity.

By inspecting lines linking to buses 4 and 5, we can easily conclude that the line 1-5  and the line 2-4 cannot be disconnected. So, variables $z_{1-5}$ and $z_{2-4}$ can be eliminated. In addition, for bus 6, we have $$z_{2-6}+z_{3-6}\geq 1,$$ according to Proposition \ref{prop_singlebus}. Consider the subset
$\mathbf{V}$ that includes buses $2,4,6$. Note that its aggregated generation capacity is 20 MW while its aggregated demand is 25 MW. Hence, based on Proposition \ref{prop_twopartition}, we have $$z_{1-2}+z_{3-6}\geq 1.$$

\noindent\textbf{Remark:}\\
\noindent $(i)$ In our numerical study, we perform a pre-processing step to analyze the network structure, heuristically generate \eqref{eq_eh_1} and \eqref{eq_eh_2}, and then include all of them in a priori fashion to the lower level problem.
For instances from large systems, we observe a significant computational time reduction. One explanation is that, although those inequalities are of simple structures, they probably reflect the sparsity nature of real power grids and effectively reduce the solution space. \\
\noindent $(ii)$ We mention that both \eqref{eq_eh_1} and \eqref{eq_eh_2} can be further strengthened by considering line capacities. So, their right-hand-sides are not necessary equal to 1. Also, there are possibly a huge number of inequalities in the form of  \eqref{eq_eh_2} that can be generated. Hence, future research directions include how to develop valid inequalities stronger than \eqref{eq_eh_1} or \eqref{eq_eh_2} and designing novel strategies for on-the-fly generation of those valid inequalities.

\section{Computational Results}
In this section, we first describe our test data sets and computational platform. Then, we present and  analyze results regarding our computational methods. Finally, we
give detailed numerical results and discuss the impact of topology control.

\subsection{Data Sets and Computational Platform}
Our test bed consists a few popular IEEE systems\cite{Christie2000Power}, including RTS-96 (24-bus),57-bus and 118-bus systems. Parameters of those systems are modified based on original data, where demands, generation and line capacities are multiplied by 1.5, 10 and 5 for RTS-96, 57-bus and 118-bus systems, respectively. Those demands are used as base loads to generate demand blocks based on Figure \ref{fig:LW_DC}. Specifically, 4 demand blocks, whose durations are 1200, 3600, 2400 and 1560 hours respectively, have 93\%, 81.7\%, 69.2\% and 59\% of the base loads respectively. Also,  wind intensities are 31.7\%, 30.6\%, 43.2\% and 30.4\% of the base wind intensity, which is adopted from  \cite{Baringo2011Wind}.

As for a generator, we partition the whole capacity into several production blocks with different prices. They are 50\%, 30\% and 20\% of its total capacity for 3 blocks or 40\%, 30\%, 20\% and 10\% for 4 blocks, respectively. Then, prices of blocks are random generated from ranges that are increasing and non-overlapping as in \cite{Baringo2011Wind}. For RTS-96 system, we have 4-6 sites that are eligible for wind farm installation. As for 57-bus and 118-bus systems, 7 to 15 sites are eligible. The total budget for wind farm installation are \$3,000 million, \$8,000 million and \$12,000 million for those 3 systems, respectively. Cost of wind generation capacity is \$0.8 million for one MW  and the fixed cost equals to the cost of 20 MW capacity for all test beds. Capacity upper bounds for wind farm installations are obtained by randomly allocating  $\frac{\textrm{total budget}}{\textrm{\$0.7 M}}$ megawatts among all eligible sites. We consider a 15-year planning period and an 8\% interest rate, which means 11.68\% is used to compute annualized  fixed cost and investment cost.  Finally, for each test bed, we randomly generate 5 instances and obtain 15 instances overall. For all of them, the value of $\kappa$ is set to 10 in our computation.

Our algorithm, including the variant with enhancement, is implemented in  C++ and tested on a  PC with Intel Core i5-45990 CPU of 3.30GHz and 4GB memory. CPLEX 12.6 \cite{CPLEX2014} is called as an MIP solver to compute master and subproblems, where $M$ is set to $10^6$ to linearize complementarity constraints. The optimality tolerance of our algorithm is set to 0.1\%, the optimality gap inside CPLEX is set to 0.01\%,  the time limit for one iteration is 600 seconds and the overall time  limit is 3,600 seconds.

\subsection{Performance of Computational Methods}
%For all 30 randomly generated instances, we test our
%
%We also report the gap before termination, defined as $\frac{UB-LB}{LB}$, for all instances.

We first select one instance of 118-bus system to illustrate the convergence of the standard reformulation and decomposition computational scheme.  Note that this system has 177 lines, which requires 177 binary line switching variables in the lower level problem. Along with hundred continuous variables and constraints for fuel-based generation, phase angles, and wind power absorption,  the lower level problem is clearly sophisticated and such scale bilevel MIP instance is assumed to be very challenging to compute. Nevertheless, our method demonstrates a quick convergence behavior as observed in Figure \ref{fig:Converge}. Only 4 iterations and about 2500 CPU seconds are sufficient to close the gap within the tolerance. Indeed, as shown in the following table, our enhanced method can further reduce the number of iterations and  the overall computational expenses.

\begin{figure}[!htbp]
\begin{center}
\includegraphics[width=0.6\textwidth]{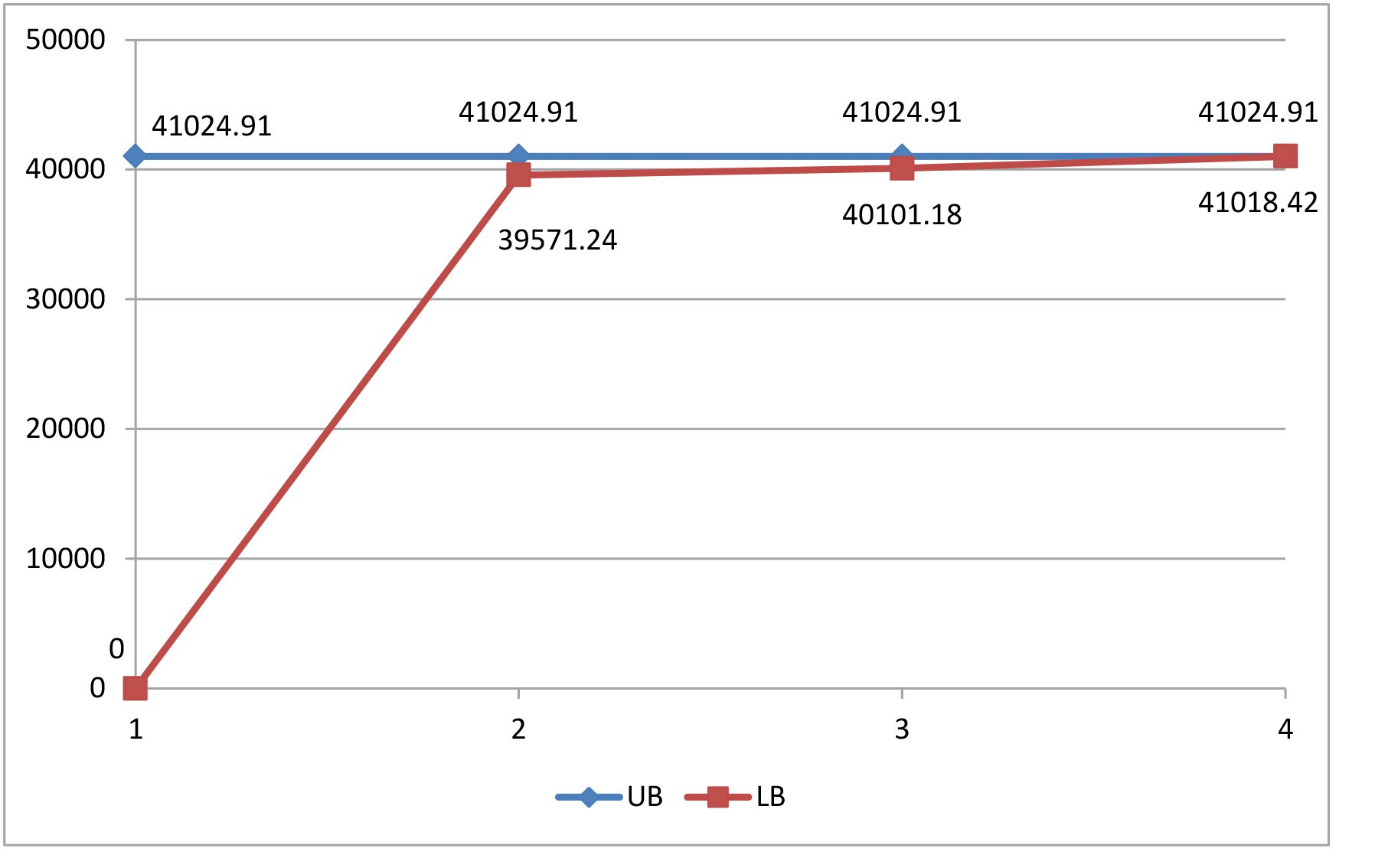}
\end{center}
\caption{Convergence over iterations}\label{fig:Converge}
\end{figure}

Table \ref{tbl:computation} reports the detailed results regarding the algorithm performance of both the standard method and its enhanced variant on all  instances.
Column ``\# Iter'' displays the number of iterations before the algorithm is terminated. Columns ``UB'', ``LB'' and ``Gap'' provide upper and lower bounds, and
the relative gap between them  at termination. Column ``CPU (s)'' gives the overall computational time in seconds. Row ``Ave.'' computes mean values of comparable columns over all instances for each test bed. Based on information presented in this table, we have the following two observations:

\noindent $(i)$ The customized reformulation and decomposition can effectively handle the challenge from binary line switching variables, and exhibits a strong capacity to compute our bilevel wind power generation capacity planning problem. As can be seen, for all instances,  it can derive optimal solutions with a few iterations in a reasonable time. Hence, we believe that it provides a fundamental platform for more advanced study, such as computing those using stochastic scenarios to capture wind or load uncertainties.\\
\noindent $(ii)$ The enhanced variant demonstrates a significant improvement over the standard one. Note that it generally can reduce the overall computational time by 50\%. For 57 and 118-bus systems, it also reduces the number necessary iterations by about 20\%. Those evidences definitely support us to pursue research in this direction by developing stronger valid inequalities and computationally friendly implementation methods.

\begin{table}[]
\centering
\caption{Computational performance}\label{table:performance}
\scalebox{0.7}{
\begin{tabular}{@{}lllllllllllll@{}}
\toprule
\multicolumn{1}{c}{\multirow{2}{*}{}}   & \multicolumn{6}{c}{Standard}                                 &  & \multicolumn{5}{c}{Enhanced}                       \\ \cmidrule(lr){3-7} \cmidrule(l){9-13}
\multicolumn{1}{c}{}                       & ID   & \# Iter & UB       & LB       & Gap    & CPU (s) &  & \# Iter & UB       & LB       & Gap    & CPU (s) \\ \midrule
\multicolumn{1}{c}{\multirow{6}{*}{24-bus}} & 1    & 2    & 10022.63 & 10022.63 & 0.00\% & 54.12     &  & 2    & 10022.63 & 10022.63 & 0.00\% & 13.24     \\
\multicolumn{1}{c}{}                       & 2    & 2    & 10022.63 & 10022.63 & 0.00\% & 42.36     &  & 2    & 10022.63 & 10022.63 & 0.00\% & 15.72     \\
\multicolumn{1}{c}{}                       & 3    & 2    & 10022.63 & 10022.63 & 0.00\% & 59.11     &  & 2    & 10022.63 & 10022.63 & 0.00\% & 29.69     \\
\multicolumn{1}{c}{}                       & 4    & 2    & 10022.63 & 10022.63 & 0.00\% & 89.1      &  & 2    & 10022.63 & 10022.63 & 0.00\% & 40.62     \\
\multicolumn{1}{c}{}                       & 5    & 2    & 10022.63 & 10022.63 & 0.00\% & 71.06     &  & 2    & 10022.63 & 10022.63 & 0.00\% & 28.69     \\
\multicolumn{1}{c}{}                       & \textbf{Ave.} & \textbf{2}    &          &          & \textbf{0.00\%} & \textbf{63.15}     &  & \textbf{2}    &          &          & \textbf{0.00\%} & \textbf{25.59}     \\ \midrule
\multirow{6}{*}{57-bus}                     & 1    & 3    & 27349.94 & 27343.69 & 0.02\% & 178.14    &  & 2    & 27349.94 & 27349.94 & 0.00\% & 48.19     \\
                                           & 2    & 2    & 27349.94 & 27349.94 & 0.00\% & 57.18     &  & 2    & 27349.94 & 27349.94 & 0.00\% & 74.54     \\
                                           & 3    & 3    & 27349.01 & 27324.15 & 0.09\% & 196.31    &  & 3    & 27349.01 & 27337.18 & 0.04\% & 104.68    \\
                                           & 4    & 3    & 27034.87 & 27013.46 & 0.08\% & 105.82    &  & 2    & 27060.15 & 27033.77 & 0.10\% & 56.17     \\
                                           & 5    & 2    & 27349.94 & 27323.98 & 0.09\% & 30.09     &  & 2    & 27349.94 & 27349.94 & 0.00\% & 20.91     \\
                                           & \textbf{Ave.} & \textbf{2.6}  &          &          & \textbf{0.06\%} & \textbf{113.51}    &  & \textbf{2.2}  &          &          & \textbf{0.03\%} & \textbf{60.9}      \\ \midrule
\multirow{6}{*}{118-bus}                    & 1    & 4    & 41024.91 & 41018.42
 & 0.02\% & 2532.17   &  & 3    & 41024.91 & 41024.91 & 0.00\% & 1033.54   \\
                                           & 2    & 2    & 40743.34 & 40710.36 & 0.08\% & 290.08    &  & 2    & 40721.1  & 40688.4  & 0.08\% & 344.18    \\
                                           & 3    & 3    & 41024.91 & 41013.84 & 0.03\% & 788.14    &  & 2    & 41024.91 & 41024.91 & 0.00\% & 299.54    \\
                                           & 4    & 3    & 39927.23 & 39887.18 & 0.10\% & 655.49    &  & 2    & 39925.01 & 39887.18 & 0.09\% & 198.36    \\
                                           & 5    & 3    & 41024.91 & 41013.39 & 0.03\% & 921.19    &  & 3    & 41024.91 & 41024.91 & 0.00\% & 741.03    \\
                                           & \textbf{Ave.} & \textbf{3}    &          &          & \textbf{0.05\%} & \textbf{1037.41}   &  & \textbf{2.4}  &          &          & \textbf{0.04\%} & \textbf{523.33}    \\ \bottomrule
\end{tabular}}
\label{tbl:computation}
\end{table}

\subsection{Impact of Topology Control}
We first demonstrate the impact of topology control on the locations of wind farms.  Figure \ref{fig:TC_24} presents wind farm installation plans for an RTS-96 24-bus instance with and without topology control. Note that 5 possible wind farm sites (at bus 7, 13, 17, 22 and 24) are remarked with circles.  If no topology control is employed, we observe in Figure \ref{figWithout} that wind farms are installed at 4 eligible sites (at bus 7, 13, 17 and 24).  Nevertheless, if topology control is implemented, as shown in Figure \ref{figWith}, a different plan is observed, where 4 sites (at bus 7, 13, 22 and 24) are installed with wind farms. Figure \ref{figWith} also shows that 6 lines, which are marked with stars, are switched off in the second demand block. In that block,  wind power absorption is 397.489 MW, higher than 395.801 MW if no topology control is implemented. Indeed, as more wind power is utilized, fuel-based generation occurred in expensive production blocks becomes less, which further leads to an expenditure reduction in market clearing.

\begin{figure}[t!h]
    \centering
    \begin{subfigure}{0.45\textwidth}
        \centering
       \includegraphics[height=2.9in]{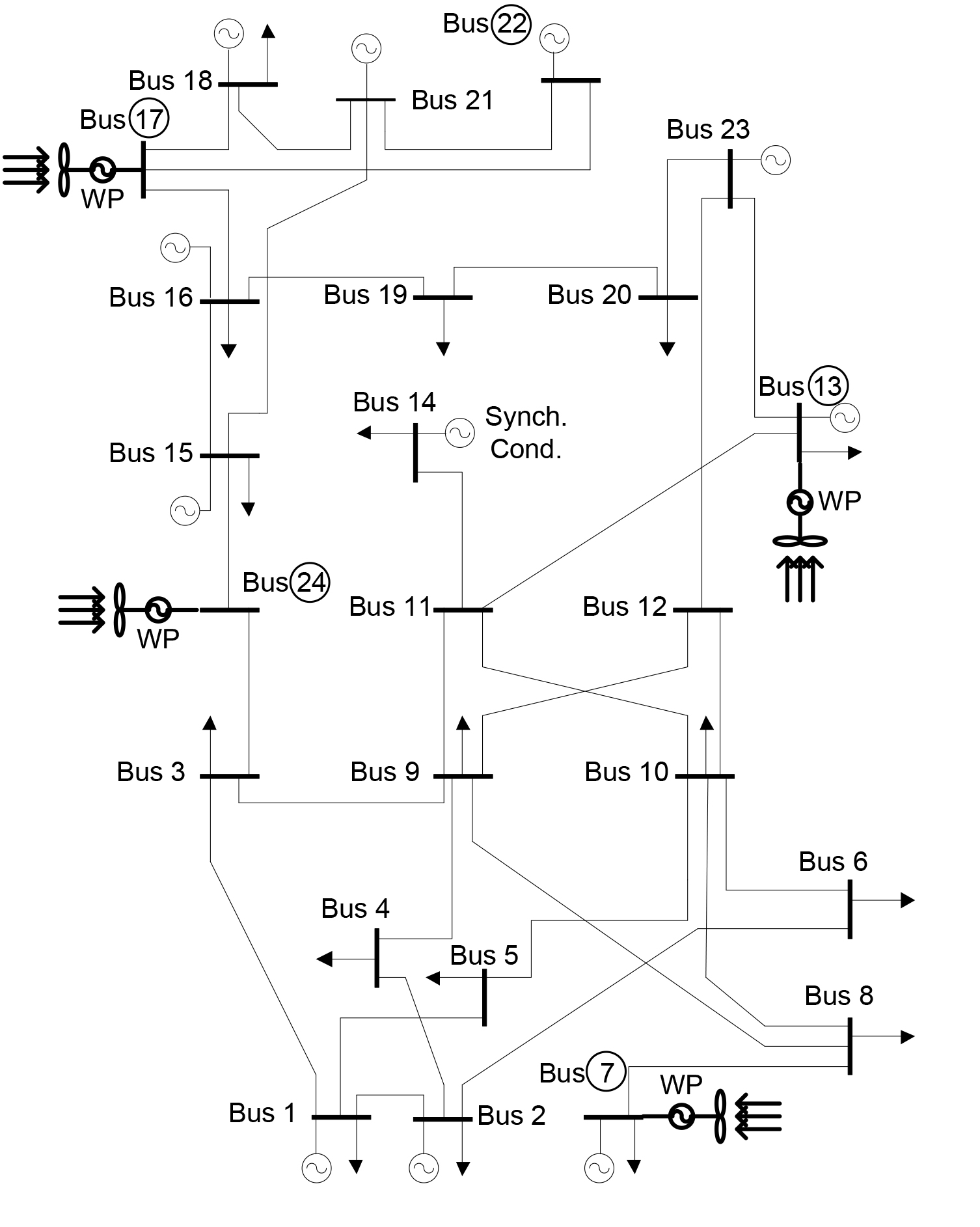}
        \caption{Installation without \protect\\ topology control}\label{figWithout}
    \end{subfigure}%
    \hspace{0.2cm}
    \begin{subfigure}{0.45\textwidth}
        \centering
      \includegraphics[height=2.9in]{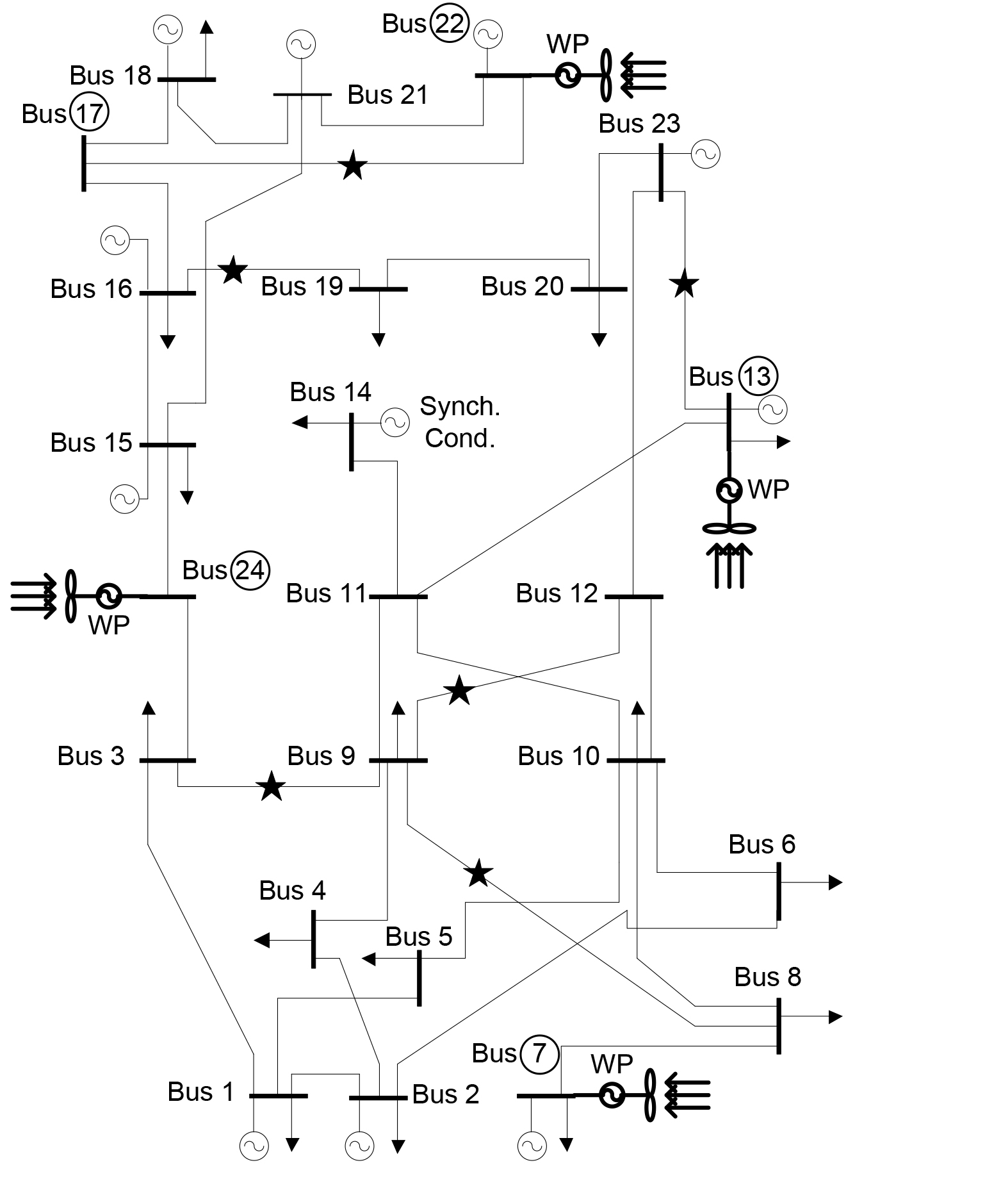}
        \caption{Installation with \protect\\ topology control}\label{figWith}
    \end{subfigure}
    \caption{Wind farm installation on RTS-96 24-bus system} \label{fig:TC_24}
\end{figure}

Table \ref{table:tc} reports the detailed results between models with and without topology control on all  instances.
Columns ``Obj.'' gives optimal value of the objective function of the upper level problem, while column ``Wind'' provides value of the first term, i.e., the weighted wind power absorption, in that objective function. Columns ``Obj-Impr.'' and ``Wind-Impr.'' compute the relative improvements due to topology control in terms of  ``Obj'' and ``Wind''. Based on information presented in this table, especially for 57-bus and 118 bus systems, we can conclude that topology control, by better placing wind farms and absorbing more wind power in dispatching, could be very helpful to reduce wind power curtailment and improve the wind penetration level.  Such improvements are non-trivial and could be more significant for large systems. As shown for instances of 118-bus system, on average almost 5\% (almost 8\% for one individual case) more wind power could be integrated.  Hence, we do believe that topology control definitely are critical and beneficial to fully make use of wind power and achieve the 20\% \cite{lindenberg200820} or 25\% \cite{davidson2016modelling} wind penetration target. Also,  as mentioned earlier, since less fuel-based generation is needed, electricity generated at a high cost will be reduced so that an economic advantage will be produced as well.

\begin{table}[t!]
\centering
\caption{Impact of topology control}\label{table:tc}
\scalebox{0.85}{
\begin{tabular}{@{}llllllllll@{}}
\toprule
\multirow{2}{*}{}    & \multirow{2}{*}{ID} & Without TC      &          &  & With TC    &          &  & \multirow{2}{*}{Obj-Impr.} & \multirow{2}{*}{Wind-Impr.} \\ \cmidrule(lr){3-4} \cmidrule(lr){6-7}
                        &                     & Obj      & Wind     &  & Obj      & Wind     &  &                       &                       \\ \midrule
\multirow{6}{*}{24-bus}  & 1                   & 10022.63 & 9672.23  &  & 10022.63 & 9672.23  &  & 0.00\%                & 0.00\%                \\
                        & 2                   & 10022.63 & 9672.23  &  & 10022.63 & 9672.23  &  & 0.00\%                & 0.00\%                \\
                        & 3                   & 10022.63 & 9672.23  &  & 10022.63 & 9672.23  &  & 0.00\%                & 0.00\%                \\
                        & 4                   & 9981.55  & 9631.15  &  & 10022.63 & 9672.23  &  & 0.41\%                & 0.43\%                \\
                        & 5                   & 10022.63 & 9672.23  &  & 10022.63 & 9672.23  &  & 0.00\%                & 0.00\%                \\
                        & \textbf{Ave.}                &          &          &  &          &          &  & \textbf{0.08\%}                & \textbf{0.09\%}                \\ \midrule
\multirow{6}{*}{57-bus}  & 1                   & 26897.45 & 26430.25 &  & 27349.94 & 26882.74 &  & 1.68\%                & 1.71\%                \\
                        & 2                   & 27349.94 & 26882.74 &  & 27349.94 & 26882.74 &  & 0.00\%                & 0.00\%                \\
                        & 3                   & 27021.15 & 26553.95 &  & 27337.18 & 26869.98 &  & 1.17\%                & 1.19\%                \\
                        & 4                   & 26615.36 & 26148.16 &  & 27033.77 & 26566.57 &  & 1.57\%                & 1.60\%                \\
                        & 5                   & 26895.53 & 26428.33 &  & 27349.94 & 26882.74 &  & 1.69\%                & 1.72\%                \\
                        & \textbf{Ave.}                &          &          &  &          &          &  & \textbf{1.22\%}                & \textbf{1.24\%}                \\ \midrule
\multirow{6}{*}{118-bus} & 1                   & 38147.19 & 36745.59 &  & 41024.91 & 39623.31 &  & 7.54\%                & 7.83\%                \\
                        & 2                   & 39050.1  & 37648.5  &  & 40688.4  & 39286.8  &  & 4.20\%                & 4.35\%                \\
                        & 3                   & 38917.13 & 37515.53 &  & 41024.91 & 39623.31 &  & 5.42\%                & 5.62\%                \\
                        & 4                   & 38363.42 & 36961.82 &  & 39887.18 & 38485.58 &  & 3.97\%                & 4.12\%                \\
                        & 5                   & 39887    & 38485.4  &  & 41024.91 & 39623.31 &  & 2.85\%                & 2.96\%                \\
                        & \textbf{Ave.}                &          &          &  &          &          &  & \textbf{4.80\%}                & \textbf{4.98\%}                \\ \bottomrule
\end{tabular}}
\end{table}

\begin{figure}[!tbp]
\begin{center}
\includegraphics[width=0.7\textwidth]{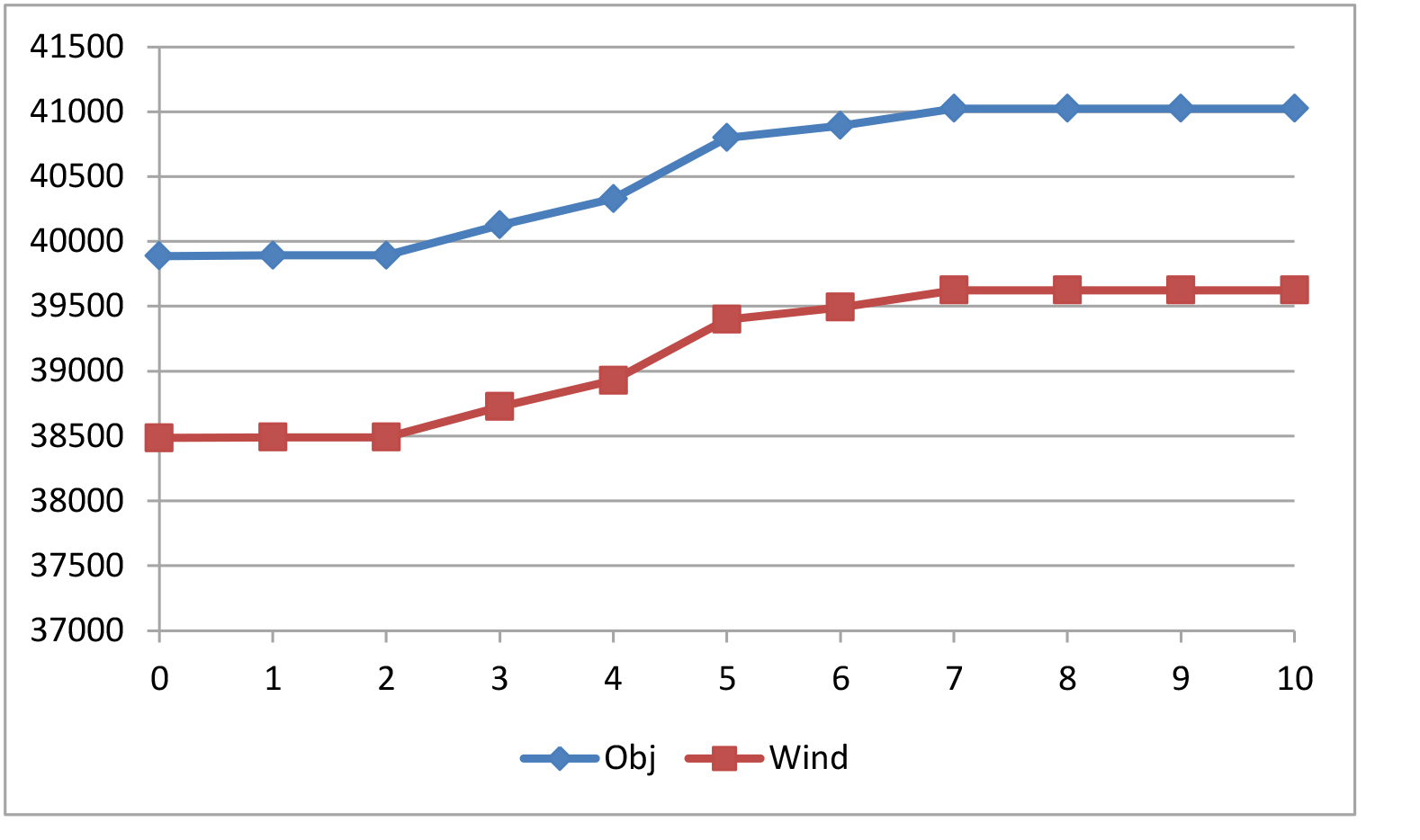}
\end{center}
\caption{Benefits of topology control on 118-bus system}
\label{fig:nbTC}
\end{figure}

In practice, switching off many existing lines could be very challenging to maintain a reliable power grid. So, it would be desired to just switching off a small number of lines for better performance. Actually, as we observe in our numerical study, it is often the case that switching off a few lines is sufficient to achieve the maximal benefits. For illustration, we perform a set of experiments on  the
5$^{\textrm{th}}$ instance of 118-bus system by imposing a cardinality constraint on the lines being switched off, i.e., $$\sum_{l\in \mathbf{L}}z_{l,t}\leq  \mathcal{K}$$
in lower level problems. We then compute the associated bilevel MIP model for different $\mathcal{K}$.
Results for $\mathcal{K}$ ranging from $0$ to $10$ are plotted in Figure \ref{fig:nbTC}, where ``Obj.''  and ``Wind'' are defined the same as those for Table \ref{table:tc}

It can be easily seen in Figure \ref{fig:nbTC} that the wind power absorption is non-decreasing with respect to $\mathcal{K}$. By allowing more lines to be switched off, up to 3\% more wind power integration can be achieved. Nevertheless, the marginal improvement becomes zero even more than 7 lines are allowed to be switched off. It suggests that by just applying topology control to a very small portion of lines, e.g., 7 out of 177 in this case, we will be able to enjoy the maximal benefits. We also would like to point out that there is no positive benefit when $\mathcal{K}\leq 2$. Such observation indicates that topology control should be applied in a way that multiple switching operations need to be coordinated, and it is unlikely to have a clear improvement from single line switching operations.

\section{Conclusions}
In this paper, we develop a novel bilevel mixed integer  optimization model to investigate wind power generation planning problem in an electricity market environment with topology control operations. As the lower level problem introduces binary variables to model line switching decisions, traditional KKT conditions based solution approach cannot be applied. To solve this challenging bilevel MIP, we customize a recent decomposition method and develop a couple of enhancement methods by making use of grid structure. Through computing instances obtained from typical IEEE test beds,  our solution methods demonstrate a strong solution capacity to this bilevel MIP model. Also, we do observe that applying topology control on a small number of lines could be very helpful to reduce wind power curtailment and improve the wind penetration level.

In the future, we would like to extend our model to consider stochastic wind and demands situations, so that we can more accurately describe random wind power generation and demand fluctuations. Certainly, it requires more advanced algorithm development. Given the significant computational improvement from the valid inequalities derived based on grid structure, we believe that stronger and more general valid inequalities and effective generation methods are worth pursuing, to further strengthen our computational capacity on practical bilevel MIP problems.

%\section*{Acknowledgment}
%Y. Wang was supported by the overseas scholarship of Chinese Scholarship Council. Y. Wang and S. Liu are sponsored by the National Natural Science Foundation of China [grant number 61573089] and B. Zeng is partially supported by U.S. National Science Foundation [CMMI-1642514].

\bibliographystyle{unsrt}
\bibliography{My_Bib}   % name your BibTeX data base

\end{document}